\documentclass[11pt]{article}
\pdfoutput=1
\usepackage{amssymb}
\usepackage{amsmath}
\usepackage{amsthm}
\usepackage{amsfonts}
\usepackage{enumerate}
\usepackage{url}
\usepackage{graphicx}
\usepackage{longtable}
\usepackage{appendix}

\newtheorem{theorem}{Theorem}[section]
\newtheorem{conjecture}[theorem]{Conjecture}

\newtheorem{corollary}[theorem]{Corollary}
\newtheorem{proposition}[theorem]{Proposition}
\newtheorem{lemma}[theorem]{Lemma}

\theoremstyle{definition} 
\newtheorem{definition}[theorem]{Definition}
\newtheorem{remark}[theorem]{Remark}

\begin{document}

\title{Paradoxical behavior in Collatz sequences}

\author{Olivier Rozier and Claude Terracol}

\date{}

\maketitle

\begin{abstract}
On the set of positive integers, we consider the iterative process that maps $n$ to either $\frac{3n+1}{2}$ or $\frac{n}{2}$ depending on the parity of $n$. The Collatz conjecture states that all such sequences eventually enter the trivial cycle $(1,2)$. In a seminal paper, Terras further conjectured that the proportion of odd terms encountered when starting with an integer $n\geq2$ is sufficient to determine its stopping time, namely, the number of iterations needed to descend below $n$. However, when iterating beyond the stopping time, there exist ``paradoxical'' sequences of finite length whose first term is unexpectedly exceeded, given the proportion of odd terms. In the present study, we show that this non-typical behavior is closely related to the Collatz conjecture. Furthermore, we find that it most likely occurs finitely many times, thus lending support to Terras' conjecture.
\end{abstract}

\hspace{10pt}


\section{Introduction}
\label{sec:intro}

According to the Collatz conjecture, whatever the positive integer $n$ taken as first term, iterating the function 

\begin{equation}\label{eq:T}
T(n) = \left\{\begin{array}{ll}
  \frac{3n+1}{2} & \mbox{if $n$ is odd,} \\
  \frac{n}{2} & \mbox{if $n$ is even,} 
\end{array}\right.
\end{equation}
gives rise to a sequence $n, T(n), T^{2}(n), \ldots$ which invariably reaches the trivial cycle $(1,2)$. As is easily seen, the rise and fall of the sequence is mainly governed by the relative proportions of odd and even terms encountered. For a sequence of $j$ iterations with $q$ odd terms (before the last term), we have a seemingly linear expression in the variable $n$ of the form
\begin{equation}\label{eq:linear}
T^{j}(n) = \frac{3^{q}}{2^{j}} \; n + E_j(n)
\end{equation}
where the fraction $\frac{3^{q}}{2^{j}}$ is called the \textit{coefficient}, hereafter noted $C_j(n)$. The rational number $E_j(n)$, called the \textit{remainder} in \cite{Ter76}, can be considered as an error term that accounts for all the increments occurring at every odd iterate. This term only depends on the parities of the successive iterates, namely, the \textit{parity vector} of $n$. We omit the general expression of $E_j(n)$ as it already appears in many papers \cite{Gar81,Lag85,Tao22,Ter76} and is superfluous in our study. When trying to predict whether $T^{j}(n)$ exceeds the first term $n$ of the sequence, it is often convenient to focus on $C_j(n)$ and to neglect the contribution of $E_j(n)$, which seems reasonable for large $n$ (see, e.g., \cite[p7]{Tao22}). Hence, comparing the ratio $\frac{q}{j}$ to the threshold value $\frac{\log 2 }{\log 3}$ is usually considered as a fairly reliable criterion in this regard \cite{Ter76}.

The aim of our investigations will be to determine to what extent the coefficient $C_j(n)$ controls whether the sequence grows or declines with respect to the first term. Since $E_j(n) \geq 0$, it is obvious that $T^{j}(n)  \geq n$ as soon as $C_j(n) \geq 1$. However, the converse does not hold. If we omit the trivial cases linked to the cycle $(1,2)$, the most simple example is the sequence of 8 iterations 
$$7, 11, 17, 26, 13, 20, 10, 5, 8$$
with coefficient $\frac{3^{5}}{2^{8}} = \frac{243}{256} < 1$ and remainder $E_8(7) = \frac{347}{256}$. We suggest calling \textit{paradoxical} those finite sequences that somehow behave unexpectedly.

\begin{definition}\label{def:par}
For positive integers $j$ and $n$ with $n>2$, the Collatz sequence $n$, $T(n)$, \ldots, $T^{j}(n)$ is \textit{paradoxical} if there hold the two conditions
\begin{enumerate}[(i)]
\item $C_j(n) = \frac{3^{q}}{2^{j}} < 1$ or, equivalently, $\frac{q}{j} < \frac{\log 2 }{\log 3}$ where $q$ is the number of odd terms among $n$, $T(n)$, \ldots, $T^{j-1}(n)$;
\item $T^{j}(n)  \geq n$. \label{cond:par_def2}
\end{enumerate}
\end{definition}

For convenience, we exclude the sequences that just repeat the trivial cycle a finite number of times from this definition.

The set of paradoxical sequences consists of two disjoint subsets, the first one for cyclic sequences ($T^{j}(n) = n$) and the second one for acyclic sequences ($T^{j}(n) > n$). The first subset is closely related to the set of nontrivial cycles for the following reason: Any Collatz sequence that starts and ends at the same integer (greater than 2) and has at least one iteration of $T$ is necessarily paradoxical. This case has already been thoroughly studied and it is known that the existence of nontrivial cycles is strongly constrained in terms of their length \cite{Cra78,Eli93} or the number of their local minima \cite{Kni26,Roz90,Sim05}. Therefore, the first subset is expected to be empty, which is consistent with the Collatz conjecture.

As for the story behind our study of the second subset, it all began about ten years ago when the second author (Terracol), a retired engineer, sent an informal manuscript to the first author. This resulted in an attempt to categorize Collatz sequences more effectively. Incidentally, we found a few examples of acyclic paradoxical sequences\footnote{Designated as ``parcours paradoxaux'' in French, they were based on a different formalism that ensured  only odd terms were considered. Hence, they proved difficult to find at once.}. A previous paper by the first author \cite{Roz17} was partly inspired by this research, though it did not mention this finding. In the present paper, we focus on these sequences alone as they appear to be largely unrecognized despite their clear connections with previous research (e.g., \cite{Alb22,Nat24}).

The existence of acyclic paradoxical sequences has undoubtedly been observed by many, although it is hard to find a publication mentioning them. Some have even assumed their nonexistence. In fact, this concept appears in a seminal paper by Terras \cite{Ter76} in a restricted and somewhat disguised form through the notions of \textit{stopping time} (entry A126241 in \cite{OEIS}) and \textit{coefficient stopping time}.

\begin{definition} \label{def:st}
Let $n$ be a positive integer.
\begin{enumerate}[(1)]
\item Its \textit{stopping time} $t(n)$ is the least integer $j$ such that $T^j(n) < n$ if it exists, otherwise $t(n) = + \infty$.
\item Its \textit{coefficient stopping time} $\tau(n)$ is the least integer $j$ such that $C_j(n) < 1$ if it exists, otherwise $\tau(n) = + \infty$. 
\end{enumerate}
\end{definition}

Since the remainder is always nonnegative, it is clear that $t(n) \geq \tau(n)$ for all $n$. Terras conjectured that equality holds for  $n \geq 2$ (Conjecture 2.9 in \cite{Ter76}), which was named Coefficient Stopping Time (CST) conjecture by Lagarias \cite[p11]{Lag85}. It is important because it implies the absence of nontrivial cycles. If indeed we assume the existence of a nontrivial cycle of length $l$, then its smallest term $n$ would have an infinite stopping time whereas $\tau(n) \leq l$.
Moreover, the CST conjecture is linked to our study through the three equivalent formulations below.
\begin{enumerate}[C1:]
\item $t(n) = \tau(n)$ for all $n \geq 2$ (CST conjecture);
\item the number of iterations in a paradoxical sequence is always larger than the stopping time of its first term;
\item the first term of a paradoxical sequence is never the smallest.
\end{enumerate}

To see this, observe that a counterexample to one of these statements would be a counterexample to the other two. The previous example of paradoxical sequence, which starts at $n=7$, is consistent with C2 and C3 because it attains the value 5 before ending at 8. One can also verify that $t(7)=\tau(7)=7$ in agreement with C1. In a sense, our approach places this conjecture in a broader context by asking how many paradoxical sequences exist, a question not addressed by Terras nor by Lagarias.

Additionally, a better understanding of paradoxical sequences is of interest for solving equations of the form $T^j(n) - n = d$ for a given nonnegative integer $d$, among which the most relevant case is $d=0$. It can be verified that all solutions for $d=1$ and $n > 3$ are due to paradoxical sequences, such as the sequence starting at $n=7$ and ending at 8 (see Appendix \ref{ap:near-cycles}).

Our main results are summarized in the following theorem. Its proof is partly based on the first author's computations and partly on numerical data from Oliveira e Silva \cite{Oli10} and Roosendaal \cite{Roo}.

\begin{theorem}
There are exactly 593 paradoxical sequences that begin with an integer lower than or equal to 4614. If there are any more of them, they must start above $2.8 \times 10^{19}$. Conversely, if there exist no others, then the Collatz conjecture is true.
\end{theorem}

Note that sequences differing in their first term or length are considered distinct, so that there are only 550 integers known to initiate a paradoxical sequence (see the computational results in Section \ref{sec:comp}).

First, we study in Section \ref{sec:poset} how the remainder term behaves with respect to the set of parity vectors by introducing a partial order on the latter. Then, we show in Section \ref{sec:div_par} that the Collatz conjecture can be derived from the assumption that there is only a finite number of paradoxical sequences. In Sections \ref{sec:prop_par} and \ref{sec:comp}, we establish various necessary conditions for a sequence to be paradoxical and verify numerically that there is none if the first term is above 4614 and below $2.8 \times 10^{19}$. Finally, we conduct in Section \ref{sec:heuristic} a heuristic analysis that leads us to conjecture  the absence of any paradoxical sequence starting above 4614.


\section{A partial order on parity vectors}
\label{sec:poset}

One may expect that paradoxical sequences are more likely to occur for large values of the remainder terms $E_j$. Therefore, it seems relevant to better understand the relationship between the remainders on an arbitrary number $j$ of iterations and their associated parity vectors of length $j$. To this aim, we need to define formally a few notions and agree on some convenient notation.

\begin{definition} \label{def:q_E_Vj}
For positive integers $j$ and $n$, let $q_j(n)$ denote the number of odd terms in the finite sequence $\left( T^{k}(n) \right)_{k=0}^{j-1}$. When appropriate, we write $q$ instead of $q_j(n)$, $E$ instead of $E_j(n)$ and $C$ instead of $C_j(n)$, for convenience, and keep in mind that it depends on $j$ and $n$. Moreover, we denote by $V_j(n) \in \left\lbrace 0, 1 \right\rbrace^j$ the \textit{parity vector} of length $j$ induced by $n$, formally defined by
$$V_j(n) = \left(  n, T(n), \ldots , T^{j-1}(n) \right) \mod 2.$$
We call \textit{ones-ratio} of a parity vector $V_j(n)$ the rational quantity $\frac{q}{j}$ which gives the proportion of 1's in $V_j(n)$.

To further simplify notation, we write any given parity vector as a \textit{binary word}, i.e., as a concatenation of characters taken from the set $\left\lbrace 0,1 \right\rbrace $.  When the same parity is repeated more than once, we remove the repetition by indicating the number of occurrences in superscript. For example, the binary word $\left\langle  1^{2} \, 0^{3} \, 1 \right\rangle$ stands for $\left\langle  1 1 0 0 0 1 \right\rangle$ and has a ones-ratio of $\frac{1}{2}$.
\end{definition}

Following an idea of the second author, we propose defining a partial order, denoted by $\preceq$, on the set of parity vectors. While this notion is relevant for our understanding of the remainder terms, it is not really necessary for proving the main result of this section, which can be established more directly.

\begin{definition} \label{def:poset}
Whenever two parity vectors $V$ and $V'$ of the same length can be written under the form $V=\left\langle  w_{1} \, 0 1 \, w_{2} \right\rangle$ and $V'=\left\langle  w_{1} \, 1 0 \, w_{2} \right\rangle$ where $w_{1}$ and $w_{2}$ are (possibly empty) binary words, we say that $V$ (strictly) \textit{precedes} $V'$, which is denoted by $V \prec V'$. Next, we extend this relation by adding transitivity: If the vectors $V$, $V'$ and $V''$ satisfy $V \prec V'$ and $V' \prec V''$, then $V \prec V''$. Similarly, we write $V \preceq V'$ if $V \prec V'$ or $V = V'$. The relation $\prec$ is consistent with the lexicographical order on binary words. Therefore, $\preceq$ is antisymmetric and induces a partial order on parity vectors. Its associated graph is directed and acyclic.
\end{definition}

The relation $V \prec V'$ between two parity vectors occurs whenever $V$ can be changed into $V'$ by repeatedly shifting 1's to the left, as though all 0's were empty spaces on an abacus. It is easy to find two distinct vectors that are incomparable through the relation $\prec$ since only vectors of the same length with the same ones-ratio may possibly be ordered. These conditions are however not sufficient when their length is at least 4, as shown by Figure \ref{fig:poset}b where the vectors $\left\langle 0110 \right\rangle$ and $\left\langle 1001 \right\rangle$ are not ordered. Similarly, in Figure \ref{fig:poset}c, the vector $\left\langle 01110 \right\rangle$ cannot be compared with any of the vectors $\left\langle 11001 \right\rangle$, $\left\langle 10101 \right\rangle$ and $\left\langle 10011 \right\rangle$.

More generally, when considering two parity vectors $V=(v_0,\ldots,v_{j-1})$ and $V'=(v'_0,\ldots,v'_{j-1})$ of length $j$, $V \preceq V'$ if and only if there hold the two conditions
\begin{enumerate}[(i)]
\item $\sum_{i=0}^{k} v_i \leq \sum_{i=0}^{k} v'_i$ \quad for $0 \leq k \leq j-2$;
\item $\sum_{i=0}^{j-1} v_i = \sum_{i=0}^{j-1} v'_i$.
\end{enumerate}
This partial order has been named \textit{unordered majorization} in \cite[\S5.D, p.\,198]{Mar11}. 
It is closely related to the \textit{majorization} or \textit{dominance} order which is widely used across various contexts involving real or integer numbers. Based on the same conditions (i) and (ii), the majorization order requires the components $v_i$ and $v'_i$ to be sorted in decreasing order, thus making it poorly suited for sequences of 0's and 1's.

\begin{figure}
\centering
\includegraphics[width=1.0\textwidth]{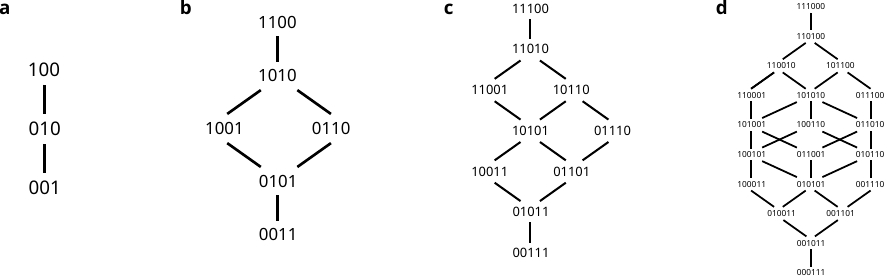}
\caption{Hasse diagrams of \textbf{(a)} the total order on parity vectors of length 3 and ones-ratio $\frac13$, and of \textbf{(b-d)} the partial orders on parity vectors of length 4, 5, 6 and ones-ratios $\frac12$, $\frac35$, $\frac12$, in that order.}
\label{fig:poset}
\end{figure}

\begin{lemma}\label{lem:poset_Ej}
If $j$, $m$ and $n$ are positive integers such that $V_j(m) \prec V_j(n)$, then $E_j(m) > E_j(n)$.
\end{lemma}

\begin{proof}
First we consider the case in which $V_j(m)$ and $V_j(n)$ can be written as $V_j(m)=\left\langle  w_{1} \, 0 1 \, w_{2} \right\rangle$ and $V_j(n)=\left\langle  w_{1} \, 1 0 \, w_{2} \right\rangle$ where $w_{1}$ and $w_{2}$ are two binary words of length $k_1$ and $k_2$, respectively. We have
$$E_{k_1+2}(m) = \frac34 E_{k_1}(m) + \frac12$$
and
$$E_{k_1+2}(n) = \frac34 E_{k_1}(n) + \frac14$$
with $E_{k_1}(m) = E_{k_1}(n) $ since $V_{k_1}(m) = \left\langle  w_{1} \right\rangle = V_{k_1}(n)$. Thus, we obtain that $E_{k_1+2}(m) > E_{k_1+2}(n)$, which in turn implies that $E_j(m) > E_j(n)$ as a result of the perfect identity between the respective parity vectors of $m$ and $n$ on the last $k_2$ iterations.

The general case easily follows by applying the transitivity of the usual order on real numbers.
\end{proof}

Recall that two positive integers have same parity vector of length $j$ if and only if they belong to the same congruence class modulo $2^j$. 
This remarkable property was a key element to establish that almost all positive integers have a finite stopping time (Definition \ref{def:st}) \cite{Eve77,Lag85,Ter76}. Moreover, this property is helpful to determine the set of positive integers $n$ that minimize or maximize the remainders $E_j(n)$.

\begin{theorem}\label{th:min_max_Ej}
For every positive integers $j$ and $n$,
\begin{equation}\label{eq:min_max_Ej}
\frac{3^q - 2^q}{2^j} \leq E_j(n) \leq \frac{3^q - 2^q}{2^q}.
\end{equation}
Moreover, the upper bound is reached if and only if $V_j(n) = \left\langle  0^{j-q} \, 1^{q} \right\rangle$, in which case $n \equiv -2^{j-q} \pmod{2^j}$.

Similarly, the lower bound is reached if and only if  $V_j(n) = \left\langle  1^{q} \, 0^{j-q} \right\rangle$, in which case $n \equiv \left( \frac23 \right)^q - 1  \pmod{2^j}$.
\end{theorem}

\begin{proof}
Let $j$ and $n$ be two positive integers. Consider the vectors $$V_{\rm min} = \left\langle  0^{j-q} \, 1^{q} \right\rangle \quad \text{and} \quad V_{\rm max} = \left\langle  1^{q} \, 0^{j-q} \right\rangle$$ where $q = q_j(n)$ as in Definition \ref{def:q_E_Vj}.

Clearly, $V_{\rm min}$ precedes any parity vector $V$ of length $j$ and ones-ratio $\frac{q}{j}$ with respect to the partial order $\preceq$, and any such parity vector $V$ precedes $V_{\rm max}$. Thus,
\begin{equation} \label{eq:V_n}
V_{\rm min} \preceq V_j(n) \preceq  V_{\rm max}.
\end{equation}

Next, consider two positive integers $n_{\rm min}$ and $n_{\rm max}$ such that 
$$ n_{\rm min} \equiv -2^{j-q} \pmod{2^j}$$
and
$$ n_{\rm max} \equiv \left( \frac23 \right)^q - 1  \pmod{2^j}.$$
By iterating the function $T$, $j$ times, from the starting values $n_{\rm min}$ and $n_{\rm max}$, it is straightforward to verify that $V_j\left( n_{\rm min} \right) = V_{\rm min}$ and $V_j\left( n_{\rm max} \right) = V_{\rm max}$.

Now, we estimate the respective remainders of $n_{\rm min}$ and $n_{\rm max}$ when $q \geq 1$:
\begin{align}
\label{eq:E_min}
E_j\left( n_{\rm min} \right) &= \frac12 \sum_{k=0}^{q-1} \left( \frac32 \right)^{k} = \left( \frac32 \right)^{q} - 1, \\
\label{eq:E_max}
E_j\left( n_{\rm max} \right) &= \frac{1}{2^{j-q}} \cdot \frac12 \sum_{k=0}^{q-1} \left( \frac32 \right)^{k} = \frac{3^q - 2^q}{2^j}.
\end{align}
From \eqref{eq:V_n}, \eqref{eq:E_min} and \eqref{eq:E_max}, we obtain both inequalities in \eqref{eq:min_max_Ej} by applying Lemma \ref{lem:poset_Ej}. Moreover, the set of integers for which the upper and lower bounds are reached are the congruence classes modulo $2^j$ of $n_{\rm min}$ and $n_{\rm max}$, respectively.

For completeness, the case $q=0$ is easy to verify as it results in $V_j(n) = \left\langle  0^{j} \right\rangle$ and $n \equiv 0 \pmod{2^j}$, so $E_j(n) = 0$.
\end{proof}

According to Theorem \ref{th:min_max_Ej}, large values of the remainder tend to occur for sequences ending with many odd terms. Therefore, one should expect most paradoxical sequences to end with an increasing trend. This may explain why finding a counterexample to the CST conjecture is difficult, if not impossible.

Another observation is that the maximal possible value of the remainder depends only on $q$ and grows exponentially fast as $q$ increases. For all sequences of $j$ iterations of the function $T$, its value is maximal when $q=j$, in which case it is close to $\left( \frac32 \right)^j$ when $j$ is large. Our next result shows that its average value grows much more slowly.

\begin{lemma}\label{lem:mean_Ej}
For every positive integer $j$, the arithmetic mean of $\left\lbrace E_j(n) \right\rbrace_{n=1}^{2^j}$ is equal to $\frac{j}{4}$.
\end{lemma}

\begin{proof}
First, recall that every parity vector of length $j$ occurs for exactly one positive integer $n \leq 2^{j}$, as was first shown by Everett \cite{Eve77} and Terras \cite{Ter76}. We then proceed by induction on $j$ and call $\mu_j$ the mean value of $E_j(n)$ for $1 \leq n \leq 2^j$.

Starting the induction at $j=1$ yields the values $E_1(1) = \frac12$ and $E_1(2)=0$ of mean $\mu_1 = \frac14$.

Next, assume $\mu_j = \frac{j}{4}$. Among all parity vectors of length $j+1$, half ends with the value 1 and the other half ends with the value 0. Thus, the arithmetic mean of the remainders corresponding to the first half is equal to $\frac{3 \mu_j + 1}{2}$, whereas it is equal to $\frac{\mu_j}{2}$ for the other half. It yields that
$$ \mu_{j+1}= \frac12 \left( \frac{3 \mu_j + 1}{2} \right) + \frac12 \left( \frac{\mu_j}{2} \right) = \mu_j + \frac14 = \frac{j+1}{4}.$$
\end{proof}

\begin{remark}
Since the function $E_j$ has period $2^j$, its mean value on $\mathbb{Z}_{\geq 1}$ is $\frac{j}{4}$, according to Lemma \ref{lem:mean_Ej}.
\end{remark}

Much work remains to be done to better specify the distribution of the remainder $E$ with respect to the parameters $j$ and $q$. Fortunately, such detailed knowledge is not necessary for the present study.


\section{From divergent to paradoxical sequences}
\label{sec:div_par}

In this section, and throughout the paper, it is important to make clear that finite sequences that differ in their first term or length are considered as distinct, no matter how many terms they have in common.

As is easily seen, a counterexample to the Collatz conjecture is either a nontrivial cycle or a sequence diverging to infinity. The first case is obviously a paradoxical sequence whose first and last terms coincide. In fact, both cases lead to paradoxical sequences, as is demonstrated by the next theorem. Our proof makes use of a lemma which directly follows from the well-developed theory of continued fractions and is closely related to Dirichlet's approximation theorem.

\begin{lemma}\label{lem:approx}
For every $\varepsilon > 0$, there are infinitely many pairs $(a,b)$ of positive integers such that
\begin{equation}\label{eq:frac_powers}
 1 - \varepsilon < \frac{3^a}{2^b} < 1.
\end{equation}
\end{lemma}

\begin{proof}
Let $\varepsilon$ be a real that verifies $0 < \varepsilon < 1$. Setting $x = \frac{\log 2}{\log 3}$, inequalities in \eqref{eq:frac_powers} are equivalent to
$$ 0 < x - \frac{a}{b}  < \frac{-\log(1-\varepsilon)}{b \log 3}.$$ 
Since $x$ is irrational, its continued fraction expansion is infinite and leads to a sequence of rational convergents $\frac{A_n}{B_n}$ with $n \geq 0$ starting as follows:
$$ 0, \; 1, \; \frac12, \; \frac23, \; \frac58, \; \frac{12}{19}, \; \frac{41}{65}, \ldots.$$
According to the theory of continued fractions,
$$ 0 < x - \frac{A_n}{B_n} < \frac{1}{B_n^{2}} \quad \text{whenever $n$ is even.}$$
We see that all pairs $(a,b)$ with $a = A_n$ and $b = B_n$ satisfy $\eqref{eq:frac_powers}$ provided that $n$ is a sufficiently large even integer. This observation is sufficient to conclude, although we have only identified a small subset of all solutions to $\eqref{eq:frac_powers}$.
\end{proof}

\begin{theorem}\label{th:div2par}
If an integer $n \geq 2$ has an infinite stopping time, then there exist infinitely many paradoxical sequences starting from integers of the form $2^k n$ with $k$ a nonnegative integer.
\end{theorem}

\begin{proof}
Assume $n \geq 2$ to be such that $T^j(n) \geq n$ for all $j$. 
We consider the set 
$$ S = \left\lbrace (a,b) \in \mathbb{Z}_{\geq 1}^{2} : 1 - \frac{1}{4n} < \frac{3^a}{2^b} < 1 \right\rbrace$$
which is infinite by Lemma \ref{lem:approx}. For an given pair $(a,b)$ in $S$, let $j$ be the least number of iterations for which $q_j(n)=a$. 

In the case $j \geq b$, we immediately obtain a paradoxical sequence starting at $n$ and ending at $T^j(n)$ with coefficient $C_j(n) = \frac{3^a}{2^j} \leq \frac{3^a}{2^b} < 1$. If this case occurs, then it yields a counterexample to the CST conjecture of Terras since $n$ has coefficient stopping time $\tau(n) \leq j$.

To address the other case $j < b$, we set $k = b-j$ and $m = 2^k \, n$, so that
$$T^b(m) - m  = \left( \frac{3^a}{2^b} - 1 \right) m + E_b(m)$$
since $T^k(m) = n$ and $V_k(m) = \left\langle  0^k \right\rangle$. For each term on the right-hand side, we have the respective lower bounds
$$\left( \frac{3^a}{2^b} - 1 \right) m > -\frac{m}{4n} = -2^{k-2} $$
and, by Theorem \ref{th:min_max_Ej},
\begin{align*}
E_b(m) = E_j(n) &\geq \frac{3^a - 2^a}{2^j} \\
 &> \frac{2^b}{2^j} \left( 1 - \frac{1}{4n} \right) - \frac{2^a}{2^j} \quad \quad \text{since $(a,b) \in S$} \\
 &\geq \frac{2^b}{2^j} \left( 1 - \frac14 \right) - \frac{2^a}{2^j} 
 = \frac{3 \cdot 2^{b-2} - 2^a}{2^j} \quad \text{as $n \geq 1$}.
\end{align*}
Combining the lower bounds yields
$$T^b(m) - m >  \frac{ 2^{b-1} - 2^a}{2^j} \geq 0$$
where the last inequality results from $b \geq a+1$. Thus, the sequence $m,\ldots,T^b(m)$ is paradoxical.

To conclude, the proof follows from the fact that we can find a sequence of the desired form for every pair $(a,b)$ in $S$ and that these sequences are all distinct, although they largely overlap.
\end{proof}

In Theorem \ref{th:div2par}, the starting integer $n$ could be the smallest element of a  cycle. If this case occurs, then the resulting paradoxical sequences start from a finite set of integers and continuously repeat the same loop, hence they are periodic and differ mostly in their length. Since the premise of this theorem is unlikely to hold, its main interest lies in its contraposition.

\begin{corollary}\label{cor:par_col}
If there are only finitely many paradoxical sequences, then the Collatz conjecture is true.
\end{corollary}

\begin{proof}
Consider the contraposition and suppose the Collatz conjecture is false. Hence, there exists a Collatz sequence of infinite length that does not contain the integer 1. If $n$ is the smallest term of the sequence, then it has an infinite stopping time. Applying Theorem \ref{th:div2par} gives infinitely many paradoxical sequences starting from integers greater than or equal to $n$ with $n>2$.
\end{proof}


\section{Properties of paradoxical sequences}
\label{sec:prop_par}

Hereafter, for $x \in \mathbb{R}$, let $\lfloor x \rfloor$ and $\lceil x \rceil$ denote the floor and ceiling of $x$, respectively. For convenience, we also introduce the following notation for a sequence of iterates.

\begin{definition} \label{def:omega}
For integers $j \geq 0$ and $n \geq 1$, let $\,\Omega_j(n)$ denote the sequence of iterates $\left( T^{k}(n) \right)_{k=0}^j $ of length $j+1$. As previously, we write $q$ for the number $q_j(n)$ of odd terms in $\Omega_{j-1}(n)$ whenever $j \geq 1$.
\end{definition}

From the linear expression
\begin{equation}\label{eq:linear2}
T^{j}(n) = C_j(n) \; n + E_j(n)
\end{equation}
for all positive integers $j$ and $n$, it is easily seen that $\Omega_j(n)$ is paradoxical if and only if
\begin{equation} \label{eq:paradox}
 0 < 1-C \leq \frac{E}{n}
\end{equation} 
where $C$ and $E$ stand for $C_j(n)$ and $E_j(n)$, respectively. Informally, this criterion suggests that paradoxical sequences are more likely to occur when the two conditions below are met:
\begin{enumerate}[(i)]
\item the coefficient $C$ is slightly smaller than 1;
\item the ratio $\frac{E}{n}$ is not too small.
\end{enumerate}
Unfortunately, little is known about the ratio $\frac{E}{n}$ which can be larger than 1 in rare cases, when a small integer $n$ gives rise to a sequence with a rapidly increasing trend. E.g., when $n=27$ and $j=45$, we have $\frac{E}{n} = 12.96\ldots$.

According to the next result where the ratio $\frac{E}{n}$ is constrained from above, one should expect both inequalities in \eqref{eq:paradox} to be quite sharp. To obtain this upper bound, we follow the method applied by Eliahou to prove \cite[Theorem 2.1]{Eli93}.

\begin{theorem}\label{th:E_n}
If $j$ and $n$ are positive integers such that $\Omega_j(n)$ is paradoxical, then
\begin{equation}\label{eq:E_n}
1-C = \frac{2^j - 3^q}{2^j} \leq \frac{E}{n} \leq \frac{\left( 3 + h^{-1}\right)^q - 3^q}{2^j}
\end{equation}
where $q$ is the number of odd terms in $\Omega_{j-1}(n)$ and $h$ is their harmonic mean.
\end{theorem}

\begin{proof}
On the one hand, the first inequality in \eqref{eq:E_n} follows from
$$ 1 \leq \frac{T^j(n)}{n} = \frac{3^q}{2^j} + \frac{E}{n} $$
when assuming $\Omega_j(n)$ paradoxical.

On the other hand, when dividing by the product of $C$ and $n$, formula \eqref{eq:linear2} becomes
$$ 1 + \frac{E}{C\,n} = \frac{T^j(n)}{C\,n} = \frac{2^j}{3^q} \, \prod_{k=0}^{j-1} \frac{T^{k+1}(n)}{T^{k}(n)}.$$
If $m_0, \ldots, m_{q-1}$ are the odd terms in $\Omega_{j-1}(n)$, then we can write
$$ 1 + \frac{E}{C\,n} = \prod_{k=0}^{q-1}\left( 1 + \frac{1}{3m_k}\right) \leq \left( 1 +  \frac{1}{3h} \right)^q $$
since the geometric mean of $1 + \frac{1}{3m_0}$, \ldots, $1 + \frac{1}{3m_{q-1}}$ is upper bounded by their arithmetic mean, which turns out to be $1 +  \frac{1}{3h}$. The second inequality in \eqref{eq:E_n} follows.
\end{proof}

As can be seen in the proof of Theorem \ref{th:E_n}, the second inequality in \eqref{eq:E_n} holds for all sequences $\Omega_j(n)$, no matter if they are paradoxical. It is far from obvious that the lower bound of the ratio $\frac{E}{n}$ in \eqref{eq:E_n} is smaller than its upper bound, so that we can derive a strong constraint on the ones-ratio $\frac{q}{j}$ that depends on $h$ (Corollary \ref{cor:ratio_par}), or, conversely, an upper bound of $h$ that depends on the number $j$ of iterations (Corollary \ref{cor:harmonic}). In practice, we often use a weakened form of these results where the harmonic mean $h$ is replaced by the smallest term of the sequence.

\begin{corollary}\label{cor:ratio_par}
Let $j$ and $n$ be positive integers. A necessary condition for $\,\Omega_{j}(n)$ to be paradoxical is 
\begin{equation}\label{eq:ratio_par}
\frac{\log 2}{\log \left( 3+h^{-1}\right) } \leq \frac{q}{j} < \frac{\log 2}{\log 3}
\end{equation}
where $h$ is the harmonic mean of the odd terms of $\,\Omega_{j-1}(n)$.
\end{corollary}

\begin{proof}
The upper bound in \eqref{eq:ratio_par} is a reformulation of the constraint $C=\frac{3^q}{2^j}<1$.
Next, from the inequality between the upper and lower bounds of $\frac{E}{n}$ in Theorem \ref{th:E_n}, we infer that
$$ 2^j \leq \left( 3 + h^{-1}\right)^q $$
which implies the lower bound in \eqref{eq:ratio_par}.
\end{proof}

The above corollary states that the ones-ratio $\frac{q}{j}$ is strongly constrained and, thus, confirms our previous claim that the coefficient $C$ must be close to its upper bound 1, particularly when the harmonic mean $h$ is large. The first inequality generalizes a similar statement regarding Collatz cycles in \cite[Theorem 2.1]{Eli93}.

\begin{corollary}\label{cor:harmonic}
Let $j$ and $n$ be positive integers, $j \geq 2$, and let $h$ be the harmonic mean of the odd terms of $\,\Omega_{j-1}(n)$. A necessary condition for $\Omega_{j}(n)$ to be paradoxical is  
\begin{equation}\label{eq:harmonic}
h \leq \frac{1}{2^r - 3} \quad \text{with } r = \frac{j}{ \left\lfloor \frac{\log 2}{\log 3} \, j \right\rfloor}.
\end{equation}
\end{corollary}

\begin{proof}
Assume that $\Omega_{j}(n)$ is paradoxical. By Corollary \ref{cor:ratio_par},
$$h \leq \frac{1}{2^\frac{j}{q} - 3}$$
and
$$ q < \frac{\log 2}{\log 3} \, j.$$
Since $q$ is an integer, we have
$$ q \leq \left\lfloor \frac{\log 2}{\log 3} \, j \right\rfloor <  \frac{\log 2}{\log 3} \, j,$$
which we can also write
$$ \frac{j}{q} \geq r > \frac{\log 3}{\log 2},$$
or equivalently,
$$ 2^\frac{j}{q} \geq 2^r > 3.$$
We obtain the upper bound \eqref{eq:harmonic} by combining our first inequality with the above ones.
\end{proof}

The latter corollary is sufficient to completely solve the cases where $j=2, 3, 4, 6, 7$ or 9. It implies either $h <0.2$ (when $j=3$), $h \leq 1$ (when $j=2, 4, 6$) or $h < 3$ (when $j=7, 9$), hence there is no solution (apart from the trivial cycle).

Another consequence of Corollary \ref{cor:harmonic} is the existence of at most finitely many paradoxical sequences of any fixed length.


\section{Computational results}
\label{sec:comp}

Various formalisms are in use for Collatz sequences regarding the operation applied on odd terms. The division by 2 can be delayed until the next iteration \cite{Alb22,Gar81, Roo} or, on the contrary, accelerated by dividing by 2 raised to the highest possible power \cite{Nat24,Tao22}. For each formalism, there are ``paradoxical'' sequences essentially defined by the same conditions. However, the fact that a given sequence behaves paradoxically at some point can depend on the choice of a formalism.

We emphasize that most of the results presented in this section are only valid when using the function $T$ defined in \eqref{eq:T}, unless otherwise stated.

\begin{remark}
When restricting to sequences starting and ending at odd integers, we obtain the same set of paradoxical sequences with most of the formalisms in use. For this reason, we will mention the number of these particular sequences of any given length.
\end{remark}

\begin{table}[ht]
\begin{center}
\begin{tabular}{|c|r|r|r|c|c|c|}
\hline
$(j,q)$ & $C$ & $N$ & $N_{\rm odd}$ & $n$ & $E$ & $d$ \\
\hline
$(8 , 5)$ & 0.949 & 5 & 0 & 7 -- 25 & 1.35 -- 2.91 & 1 -- 2\\
$(27 , 17)$ & 0.962 & 50 & 12 & 164 -- 885 & 7.27 -- 58.68 & 1 -- 26\\
$(46 , 29)$ & 0.975 & 231 & 56 & 91 -- 4611 & 3.24 -- 273.37 & 1 -- 188\\
$(54 , 34)$ & 0.925 & 2 & 0 & 432 -- 864 & 33.06 -- 66.13 & 1 -- 2\\
$(65 , 41)$ & 0.988 & 244 & 62 & 73 -- 4547 & 7.83 -- 341.28 & 7 -- 292\\
$(73 , 46)$ & 0.938 & 56 & 18 & 487 -- 4614 & 30.99 -- 343.46 & 1 -- 63\\
$(92 , 58)$ & 0.951 & 5 & 0 & 3567 -- 4551 & 251.05 -- 344.14 & 65 -- 125\\
\hline
\end{tabular}
\caption{For each pair $(j,q)$, first digits of $C=\frac{2^j}{3^q}$, number $N$ of paradoxical sequences of the form $\Omega_j(n)$ with $n \leq 10^9$ and ones-ratio $\frac{q}{j}$, minimal and maximal values of $n$, of the remainders $E$ and of the differences $d = T^j(n)-n$. The column $N_{\rm odd}$ states how many of these sequences start and end at odd integers. All pairs $(j,q)$ for which $N=0$ are omitted.}
\label{tab:paradox}
\end{center}
\end{table}

First, we conducted a computational search for all paradoxical sequences whose starting integer is at most $10^9$. We found exactly 593 such sequences and provide in Table \ref{tab:paradox} a summary of the numerical results (see Appendix \ref{ap:comp_data} for an exhaustive list). Only seven pairs $(j,q)$ occur, leading to five distinct ratios $\frac{q}{j}$ which are mostly (lower) convergents or semiconvergents of the continued fraction expansion of $\frac{\log 2}{\log 3}$. Therefore, the corresponding coefficients $C=\frac{3^q}{2^j}$ are close to 1 and the relation
$$q = \left\lfloor \frac{\log 2}{\log 3} \, j \right\rfloor$$
always holds.

A closer look at the iterates reveals that all sequences in Table \ref{tab:paradox} contain either the value 11 when $j=8$, or the value 103 otherwise. Consequently, most of them overlap. There are 138 occurrences of a sequence $\Omega_j(n)$ starting and ending at even integers like $\Omega_{8}(18) = (18,\,9,\,14,\,7,\,11,\,17,\,26,\,13,\,20)$, in which case $\Omega_j(\frac{n}{2})$ is also paradoxical and shares $j$ terms with $\Omega_j(n)$. It can even happen that several paradoxical sequences start from the same integer. For example, the Collatz sequences starting at 859 are paradoxical after 46, 65 and 73 iterations, reaching the integers 890, 911 and 866 in that order. As a result, there are only 550 distinct integers $n$ initiating the sequences in Table \ref{tab:paradox}. Moreover, it is striking that \textit{near-cycles} (i.e., sequences of the form $\Omega_j(n)$ with $T^j(n) = n+1$) occur 20 times in Table \ref{tab:paradox}, suggesting that there may be a sort of self-avoiding property in Collatz sequences before reaching the value 1 (see Appendix \ref{ap:near-cycles} for further insights on near-cycles) .

Next, one might ask whether the results in Table \ref{tab:paradox} are still exhaustive without the condition $n \leq 10^9$. To address this question, the numerical data provided independently by Oliveira e Silva \cite{Oli10} and Roosendaal \cite{Roo} will prove to be helpful. Two particular metrics will be  considered, namely the \textit{delay} and the \textit{maximum excursion}, both related to the dynamics of Collatz sequences. Note that various terms are in use to refer to them\footnote{The delay is also called the \textit{total stopping time} in \cite{Kon10,Lag85,Roz17}, whereas the maximum excursion is called the \textit{path} in \cite{Roo}}.

\begin{definition} \label{def:mex}
Let $n$ be a positive integer. 
\begin{enumerate}[(1)]
\item The \textit{delay} $d_{\text{\tiny \rm T}}(n)$ is the least integer $j$ such that $T^j(n) = 1$ if it exists, otherwise $d_{\text{\tiny \rm T}}(n) = + \infty$.
\item The \textit{maximum excursion} $M_{\text{\tiny \rm T}}(n)$ is the greatest term in the sequence $\left( T^{k}(n) \right)_{k=0}^{\infty}$ if it exists, otherwise $M_{\text{\tiny \rm T}}(n) = + \infty$. 
\end{enumerate}
\end{definition}

In \cite{Oli10}, Oliveira e Silva computed the table of maximum excursion records for positive integers with a greater maximum excursion than any smaller integer (entry A006884 in \cite{OEIS}). Roosendaal made a similar computation for the delay records (entry A006877 in \cite{OEIS}), but using a different formalism in which the function $T$ is replaced by the function
\begin{equation}\label{def:Col}
{\rm Col}(n) = \left\{\begin{array}{ll}
  3n+1 & \mbox{if $n$ is odd,} \\
  \frac{n}{2} & \mbox{if $n$ is even,} 
\end{array}\right.
\end{equation}
whose dynamics is mostly the same, albeit slower. The corresponding delay function, denoted by $d_{\text{\tiny \rm Col}}$, is defined in the same way as $d_{\text{\tiny \rm T}}$ when ${\rm Col}$ is substituted for $T$. For every $n \geq 2$ such that $d_{\text{\tiny \rm T}}(n) < + \infty$, it is easily seen that
\begin{equation}\label{eq:dc}
d_{\text{\tiny \rm Col}}(n) = j + q 
\end{equation}
where $j = d_{\text{\tiny \rm T}}(n)$ and $q=q_j(n)$.

\begin{theorem}\label{th:par_length}
There are exactly 593 paradoxical sequences $\Omega_{j}(n)$ with $n \leq 4614$ and $j \leq 92$. There are no others in either of the following two cases:
\begin{itemize}
\item $93 \leq j \leq 301\,993$;
\item $4615 \leq n \leq 2.8 \times 10^{19}$.
\end{itemize}
\end{theorem}

\begin{proof}
Consider a paradoxical sequence $\Omega_{j}(n)$. 
Setting $n_0 = 10^9$, we first assume $n \leq n_0$. As we tested numerically all sequences in this case, we know that $\Omega_{j}(n)$ is among the 593 sequences in Table \ref{tab:paradox}, so $j \leq 92$ and $n \leq 4614$.

It remains to consider the case $n > n_0$. Let $m$ be the smallest term of $\Omega_{j-1}(n)$ and let $h$ be the harmonic mean of the $q$ odd terms of $\Omega_{j-1}(n)$. Observe that
$$ n_0 < n \leq  T^j(n) \leq M_{\text{\tiny T}}(m)$$
since $T^j(n)$ appears in $\Omega_j(m)$.

From the table of maximum excursion records in \cite[p198]{Oli10}, we find that $m_0 = 113\,383$ is the smallest integer that satisfies $M_{\text{\tiny T}}(m_0) \geq n_0$. Therefore, we have  $h \geq m \geq m_0$. Corollary \ref{cor:harmonic} gives the upper bound
$$ h \leq H(j)= \frac{1}{2^{r(j)}-3}  \quad \text{with $r(j) = \frac{j}{ \left\lfloor \frac{\log 2}{\log 3} \, j \right\rfloor}$}.$$ 
A straightforward computation shows that $j_0 = 1539$ is the smallest integer greater than 1 for which $H(j_0) \geq m_0$, thus implying $j \geq j_0$.
By Corollary \ref{cor:ratio_par},
$$ q \geq q_0 = \left\lceil \frac{j_0 \log 2}{\log\left( 3+m_0^{-1}\right) } \right\rceil = 971.$$
Observing that $d_{\text{\tiny \rm T}}(n) \geq j$, equation \eqref{eq:dc} implies
$$d_{\text{\tiny \rm Col}}(n) \geq j+q \geq j_0+q_0 = 2510.$$
Considering the highest delay in the table of delay records for the function Col in \cite{Roo}, which is currently 2456 and was reached by an integer $n_1$ slightly greater than $2.8 \times 10^{19}$, we obtain
$$T^j(n) \geq n > n_1 > 2.8 \times 10^{19}.$$
Then, using the table of maximum excursion records in \cite{Oli10} once again reveals that $m_1 = 23\,035\,537\,407$ is the smallest integer for which $M_{\text{\tiny T}}(m_1) > n_1$, and hence, $m \geq m_1$.

By Corollary \ref{cor:harmonic}, $H(j) \geq h \geq m$. Using a computer, one verifies that $j_1=301\,994$ is the smallest integer greater than 1 for which $H(j_1) \geq m_1$, so  necessarily $j \geq j_1$.
\end{proof}

These computational results can also be used to further support Terras' CST conjecture, which asserts that the stopping time $t(n)$ and the coefficient stopping time $\tau(n)$ are equal for all $n \geq 2$ (see Definition \ref{def:st}). This conjecture is known to hold true up to $n = 250\,000$, as shown in \cite[p252]{Ter76}. Garner mentions in \cite[p20]{Gar81} having extended the verification up to $n = 1\,150\,000$. He used the function Col which allows a similar formulation of the conjecture.

\begin{corollary}
For every integer $n$ such that $2 \leq n \leq 2.8 \times 10^{19}$, $t(n) = \tau(n)$.
\end{corollary}

\begin{proof}
Consider the contraposition and suppose $t(n) \neq \tau(n)$ for some $2 \leq n \leq 2.8 \times 10^{19}$. Necessarily, $t(n) > \tau(n)$. Setting $j=\tau(n)$ yields  $C_j(n) < 1$ whereas $T^j(n) \geq n$. In other words, the sequence $\Omega_j(n)$ is paradoxical. Theorem \ref{th:par_length} implies $n \leq 4614$, contradicting the numerical results of Terras and Garner.
\end{proof}

\begin{remark}
Another approach to the CST conjecture is to consider sequences with a fixed number of local maxima in order to rule out various generic shapes of parity vectors when searching for counterexamples. For instance, it is straightforward to show that acyclic paradoxical sequences with a single local maximum don't exist (see Appendix \ref{ap:1-nycle}). This method, already applied in \cite{Sim05} to address the difficult question of nontrivial cycles, was suggested to us by Simons during the revision process.
\end{remark}

\begin{remark}
We conducted a similar computational investigation for paradoxical sequences associated with the function $\rm Col$ defined in \eqref{def:Col} and found 1541 occurrences whose starting numbers range from 7 to 9229. If we denote by $j$ and $q$ the number of even and odd iterates, respectively, before reaching the last term, then we obtain the same pairs $(j,q)$ as in Table \ref{tab:paradox} plus the extra pairs $(16,\,10)$ and $(130,\,82)$. There are 36 occurrences for which the difference $d$ between the first and last term is equal to 1. The largest value of $d$ is 584, obtained for a sequence starting at 8648 and ending at 9232.
\end{remark}

\begin{remark}
The notion of paradoxical sequence is also relevant to several variants of the Collatz function, sometimes with a greater variety of dynamic behaviors. One interesting example occurs when positive integers $n$ are mapped to either $\frac{5n+1}{2}$ or $\frac{n}{2}$ depending on the parity of $n$. In addition to the cycle $(1,3,8,4,2)$, two more cycles of length 7 exist, starting from 13 and 17, while most sequences are expected to diverge to infinity \cite{Kon10}. These cycles generate infinitely many paradoxical sequences with a cyclic behavior that start from a finite set of integers. Furthermore, there is numerical evidence that the set of integers initiating an acyclic paradoxical sequence is infinite in this variant. We found 48\,438 of them starting below $10^7$.
\end{remark}


\section{Heuristic analysis}
\label{sec:heuristic}

In the proof of Theorem \ref{th:par_length}, we used the lower bound on the first term $n$ to increase the lower bound on the number $j$ of iterations, and vice versa. If one could extend the tables of delay and maximum excursion records, it is likely that these bounds could be further improved. The recent progress in the computational verification of the Collatz conjecture made by Barina \cite{Bar21} did not permit to extend the table of delay records as the method used was primarily targeting the stopping time and the delay was not computed. Nevertheless, if we extrapolate from this process, we may assume that it can be indefinitely repeated, ruling out the existence of a paradoxical sequence starting from a large integer.

\begin{conjecture}\label{conj:par_seq}
There is no paradoxical sequence for the function $T$ whose first term is greater than 4614.
\end{conjecture}

Observe that this conjecture is stronger than both the Collatz and the CST conjectures, according to Corollary \ref{cor:par_col}.

To further support Conjecture \ref{conj:par_seq}, we refer to several heuristic approaches of the Collatz conjecture where upper bounds have been suggested for the delay  (often called \textit{total stopping time}) and the maximum excursion of a positive integer (see Definition \ref{def:mex}). We propose starting with the weak formulation below regarding the delay as it is sufficient for our purposes.

\begin{conjecture}\label{conj:delay}
There is a positive constant $\alpha$ such that, for all $n \geq 2$,
\begin{equation}\label{eq:ub_delay}
d_{\text{\tiny T}}(n) \leq \alpha \, \log n.
\end{equation}
\end{conjecture}

It is straightforward to derive from Conjecture \ref{conj:delay} that there is another positive constant $\beta$ such that, for all $n \geq 2$,
\begin{equation}\label{eq:ub_mex}
M_{\text{\tiny T}}(n) \leq n^{\beta}.
\end{equation}

According to the heuristic approaches presented in \cite{Kon10,Roz17} (see also \cite[p11]{Lag10}), these upper bounds are expected to hold with $\alpha = 41.677\ldots + \varepsilon$ and $\beta = 2+\varepsilon$ for every $\varepsilon > 0$ and for every $n$ sufficiently large (i.e., for $n \geq N_\varepsilon$ where $N_\varepsilon$ depends only on $\varepsilon$). Empirically, they are satisfied with $\alpha=37$ and $\beta=2.6$ for every $n \geq 2$ tested so far, as reported in \cite[Tables 1 and 3]{Kon10} (see also ``Gamma" and ``Path" records in \cite{Roo} for the latest update on numerical data).

Assuming Conjecture \ref{conj:delay} is true, one can study its implications regarding the number of paradoxical sequences. Hence, let us consider a paradoxical sequence $\Omega_j(n)$ with $n>2$. Let $h$ denote the harmonic mean of the $q$ odd terms of $\,\Omega_{j-1}(n)$ and let $m$ be the smallest of these terms. Since  $m$ is odd and greater than 1, we have
$$ 3 \leq m \leq n \leq T^j(n) \leq M_{\text{\tiny T}}(m) \leq m^{\beta}$$
by applying \eqref{eq:ub_mex} to the integer $m$, so
\begin{equation} \label{eq:ub_mex2}
 h \geq m \geq n^{\frac{1}{\beta}}.
\end{equation}
Next, the upper bound \eqref{eq:ub_delay} implies
\begin{equation} \label{eq:delay2}
 j < d_{\text{\tiny T}}(n) \leq \alpha \, \log n.
\end{equation}
Moreover, by Corollary \ref{cor:ratio_par},
$$\frac{\log 2}{\log \left( 3+h^{-1}\right) } \leq \frac{q}{j} < \frac{\log 2}{\log 3},$$
which we can write
\begin{equation} \label{eq:par_seq_approx}
 0 < \frac{j}{q} - \frac{\log 3}{\log 2} \leq \frac{\log \left( 1+\frac{1}{3h}\right)}{\log 2} < \frac{1}{3h \log 2}.
\end{equation}
Combining \eqref{eq:ub_mex2}, \eqref{eq:delay2} and \eqref{eq:par_seq_approx} yields
\begin{equation} \label{eq:heuristic}
0 < \frac{j}{q} - \frac{\log 3}{\log 2} < \frac{e^{-\frac{j}{\alpha \beta}}}{3 \log 2}
\end{equation}
and, multiplying by $q\log 2$,
\begin{equation} \label{eq:heuristic2}
0 < j \log2 - q \log 3 < \frac{q}{3} \, e^{-\frac{j}{\alpha \beta}}.
\end{equation}

Clearly, the inequalities in \eqref{eq:heuristic} and \eqref{eq:heuristic2} have finitely many solutions $(j,q)$. This can be shown using either Baker's theorem on linear forms in logarithms \cite[p22]{Bak75} or the effective result of Rhin \cite{Rhi87}, which was obtained from a computer-aided proof and appears in several papers \cite{Roz90,Sim05} dealing with Collatz cycles.

\begin{proposition} {\em (Rhin)} \label{prop:approx_rhin}
If $u_0$, $u_1$, $u_2$ are rational integers such that $$H=\max\left( |u_1|,|u_2| \right) \geq 2,$$ then the form $\Lambda = u_0 + u_1 \log 2 + u_2 \log 3$ verifies
\begin{equation} \label{eq:approx_rhin}
 |\Lambda| \geq H^{-13.3}.
\end{equation}
\end{proposition}

Setting $u_0 = 0$, $u_1 = j$ and $u_2 = -q$, the lower bound \eqref{eq:approx_rhin} becomes
$$ \left| j \log 2 - q \log 3 \right| \geq j^{-13.3}$$
and, combined with \eqref{eq:heuristic2}, yields the condition
$$ j^{14.3} \; e^{-\frac{j}{\alpha \beta}} > \frac{3j}{q} > \frac{3 \log 3}{\log 2} = 4.754\ldots.$$
It follows that the number $j$ of iterations is necessarily upper bounded since the leftmost expression could otherwise be made arbitrarily close to zero, contradicting the above condition. For instance, when we set $\alpha = 42$ and $\beta=3$, the bound $j < 17\,397$ holds. 

In summary, heuristic arguments and numerical evidence support Conjecture \ref{conj:delay} which, if true, implies that the number of paradoxical sequences is finite.

\section*{Acknowledgements}
The authors are grateful to Christian Boyer who played a decisive role by connecting them. They also would like to thank the referees for their suggestions that helped improve the presentation.

\appendix
\appendixpage

\section{Sequences with a unique local maximum}
\label{ap:1-nycle}

A paradoxical sequence with a unique local maximum consists of an increasing subsequence followed by a decreasing one. The first term would necessarily be the smallest, which contradicts C3 (cf. Section \ref{sec:intro}). Therefore, it would be a counterexample to the CST conjecture. In the cyclic case, this sequence is usually called a \textit{1-cycle} or  \textit{circuit}, which is known not to exist apart from the trivial cycle. The proof requires results from transcendental number theory like Proposition \ref{prop:approx_rhin} (see, e.g., \cite{Kni26,Roz90,Sim05}). Here, we also rule out the acyclic case using elementary only arguments.

\begin{lemma}
There is no paradoxical sequence whose parity vector is of the form $\left\langle  1^{q} \, 0^{j-q} \right\rangle$ with $j \geq q \geq 1$.
\end{lemma}

\begin{proof}
Assume the existence a paradoxical sequence $\Omega_j(n)$ of parity vector $\left\langle  1^{q} \, 0^{j-q} \right\rangle$ with $n>2$ and $j \geq q \geq 1$. The last term verifies $T^j(n) = n+d$ where $d \geq 0$. Since there is no circuit, the case $d=0$ is already settled.

We only need to consider the case $d \geq 1$. By Theorem \ref{th:min_max_Ej}, we have the equation
\begin{equation}\label{eq:Tjn}
T^j(n) = \frac{3^q}{2^j} n + \frac{3^q-2^q}{2^j} = n + d
\end{equation}
together with the congruence
$$ n \equiv -1 \pmod{2^q}.$$
Setting $n = 2^q \, k -1$ for some positive integer $k$, the equation \eqref{eq:Tjn} becomes
$$ \left( 3^q - 2^j\right) k = 2^{j-q} (d-1) + 1 \geq 1.$$
Hence, $3^q - 2^j > 0$ which contradicts our initial assumption that $\Omega_j(n)$ is paradoxical.
\end{proof}

\section{The equation $T^j(n) - n = 1$}
\label{ap:near-cycles}

Little is known about \textit{near-cycles} of the form $T^j(n) = n + 1$. As with nontrivial cycles, we should expect this equation to have only finitely many solutions with $n>1$. In Lemma \ref{lem:near-cycles}, we show that all solutions with $n > 3$ are paradoxical sequences. It turns out that this is a consequence of the lower bound in Lemma \ref{lem:knight_ellison}. This lower bound appears in a less general form in the recently published paper of Knight \cite[Lemma 3.3]{Kni26}. It was directly derived from a seemingly forgotten result of Ellison \cite[Theorem 3]{Ell71}. In the present context, the main interest of \eqref{eq:knight_ellison} is that it is stronger than Rhin's Proposition \ref{prop:approx_rhin} for small $j$ and $q$. This leads to an easy proof of Lemma \ref{lem:near-cycles}.

\begin{lemma} \label{lem:knight_ellison}
For every positive integers $j$ and $q$ with $q>12$, there holds
\begin{equation} \label{eq:knight_ellison}
 \left| 2^j - 3^q \right| > \frac{2.56^q}{2}.
\end{equation}
\end{lemma}

\begin{proof}
Let $\Delta = 2^j - 3^q$ and let $\rho = \log_2 3$. First, we assume $q>18$ and distinguish between three cases.

\noindent
\textbf{Case 1}: $\Delta > 0$, which implies $j > \rho q > 28$.
According to Theorem 3 in Ellison's paper \cite{Ell71}, we can write
$$ \left| \Delta \right| > 2^j e^{-\frac{j}{10}} = c^j > c^{\rho q} = A^q$$
with $c = 2 e^{-\frac{1}{10}} = 1.809\ldots$ and $A = c^{\rho} = 2.560278\ldots$.

\noindent
\textbf{Case 2}: $2^j < 3^q < 2^{j+1}$, which implies $j > \rho q - 1 > 27$. By using Ellison's result once again, there holds
$$ \left| \Delta \right| > c^j > c^{\rho q - 1} = \frac{A^q}{c} > \frac{A^q}{2}.$$

\noindent
\textbf{Case 3}: $3^q > 2^{j+1}$, which implies $2^j < \frac12 3^q$. It yields
$$ \left| \Delta \right| = 3^q - 2^j > \frac{3^q}{2} > \frac{A^q}{2}.$$

For completeness, we verified numerically that the lower bound in \eqref{eq:knight_ellison} always holds for $13 \leq q \leq 18$.
\end{proof}

\begin{lemma} \label{lem:near-cycles}
Let $j$ and $n$ be positive integers, $n \neq 1$, for which $T^j(n) - n = 1$. Then, the sequence $\Omega_j(n)$ is paradoxical unless $(j,n)=(3,3)$.
\end{lemma}

\begin{proof}
Using the notations of Definition \ref{def:q_E_Vj}, we write 
$$n + 1 = T^j(n) = \frac{3^q}{2^j} \, n + E.$$
By Theorem \ref{th:min_max_Ej},
$$E \geq \frac{3^q-2^q}{2^j}$$
which by a short calculation implies that
$$ \left( 3^q - 2^j \right) (n + 1) \leq 2^q$$
or, equivalently,
$$ 3^q - 2^j \leq \frac{2^q}{n+1}.$$

First, consider the general case $q>12$. If we assume $3^q > 2^j$, then by Lemma \ref{lem:knight_ellison},
$$ \frac{2.56^q}{2} < \frac{2^q}{n+1}$$
which obviously cannot be true. Hence, necessarily $3^q < 2^j$ and $\Omega_j(n)$ is paradoxical.

Next, to address the case $q \leq 12$, we observe that the only solutions to 
$$ 0 < 3^q - 2^j \leq 2^{q-1} \quad \text{with $q \leq 12$}$$
are $(j,q)=(1,1)$ and $(j,q)=(3,2)$. Then, it is easy to verify that the only nonparadoxical solution to $T^j(n) - n = 1$ is $(j,n) = (3,3)$, assuming $n>1$.
\end{proof}

According to Lemma \ref{lem:near-cycles}, if we consider the computational results in Section \ref{sec:comp} together with Conjecture \ref{conj:par_seq}, then we should expect  the solutions to the equation $T^j(n) - n = 1$ with $n>1$ to be the 21 pairs $(j,n)$ enumerated below:

\noindent
$(3,3) \; (8,7) \; (8,9) \; (8,19) \; (8,25) \; (27,166) \; (27,250) \; (27,333) \; (27,376) \; (46,91) \; (46,121)$\\
$ \; (46,243) \; (46,667) \; (46,889) \; (54,432) \; (73,487) \; 
(73,649) \; (73,865) \; (73,1153) \; (73,2307)$\\
$(73,3643)$.

Surprisingly, there is no solution with $j=65$, even though this length is the most prevalent among the set of paradoxical sequences, according to Table \ref{tab:paradox}.

If we omit the condition $n>1$, then every pair $(j,n)$ with $n=1$ and $j$ odd is also a solution because $T^j(1) = 2$.

\section{Computational data}
\label{ap:comp_data}
For each admissible length $j$ and number $q$ of odd terms, we provide the list of integers $n$ initiating an acyclic paradoxical sequence for the function $T$.

\vspace{10pt}


\begin{longtable}{p{.10\textwidth} p{.84\textwidth}}
$(j,q)$ & $n$ \\
\hline
$(8 , 5)$   & 7, 9, 18, 19, 25 \\
$(27 , 17)$ & 164, 165, 166, 171, 231, 248, 250, 257, 258, 259, 328, 330, 332, 333,
             342, 343, 345, 347, 376, 386, 387, 388, 389, 390, 391, 432, 436, 437,
             440, 442, 580, 581, 582, 584, 586, 587, 588, 589, 656, 660, 661, 664,
             666, 864, 872, 874, 880, 881, 884, 885 \\
$(46 , 29)$ & 91, 121, 242, 243, 313, 322, 323, 327, 411, 415, 417, 428, 429, 430, 
             457, 459, 462, 463, 470, 471, 484, 485, 486, 543, 553, 568, 570, 572,
             573, 609, 617, 619, 623, 626, 627, 644, 645, 646, 649, 654, 655, 667,
             684, 685, 686, 689, 690, 691, 694, 695, 697, 706, 707, 856, 858, 859,
             860, 861, 864, 865, 872, 873, 874, 879, 880, 881, 884, 885, 889, 898,
             899, 903, 1028, 1029, 1030, 1032, 1034, 1035, 1036, 1037, 1041,
             1042, 1043, 1046, 1047, 1060, 1061, 1062, 1136, 1140, 1141, 1144,
             1146, 1312, 1320, 1322, 1328, 1332, 1333, 1345, 1348, 1349, 1350,
             1355, 1544, 1546, 1548, 1549, 1552, 1553, 1556, 1557, 1560, 1562,
             1564, 1565, 1566, 1570, 1571, 1579, 1592, 1593, 1594, 1596, 1597,
             1728, 1744, 1748, 1749, 1760, 1762, 1763, 1768, 1770, 1777, 1793,
             1796, 1797, 1798, 1806, 1807, 1984, 2000, 2004, 2005, 2016, 2018,
             2019, 2024, 2026, 2033, 2272, 2280, 2282, 2288, 2289, 2292, 2293,
             2305, 2320, 2324, 2325, 2328, 2330, 2336, 2344, 2346, 2348, 2349,
             2350, 2352, 2356, 2357, 2369, 2390, 2391, 2392, 2394, 2396, 2397,
             2409, 2415, 2427, 3008, 3024, 3028, 3029, 3040, 3042, 3043, 3048,
             3050, 3057, 3073, 3456, 3488, 3496, 3498, 3520, 3524, 3525, 3526,
             3536, 3540, 3541, 3554, 3555, 3586, 3587, 3592, 3594, 3596, 3597,
             3601, 3612, 3613, 3614, 3623, 3641, 4544, 4560, 4564, 4565, 4576,
             4578, 4579, 4584, 4586, 4593, 4610, 4611 \\
$(54 , 34)$ & 432, 864 \\ 
$(65 , 41)$ & 73, 231, 235, 313, 411, 415, 417, 457, 459, 462, 463, 470, 471, 543,
             553, 609, 617, 619, 623, 626, 627, 639, 811, 815, 822, 823, 825, 830,
             831, 834, 835, 859, 865, 871, 873, 879, 889, 1017, 1033, 1081, 1083,
             1086, 1087, 1097, 1106, 1107, 1113, 1127, 1287, 1289, 1298, 1299,
             1308, 1309, 1310, 1319, 1334, 1335, 1519, 1526, 1527, 1537, 1550,
             1551, 1563, 1707, 1711, 1718, 1719, 1730, 1731, 1745, 1746, 1747,
             1758, 1759, 1767, 1778, 1779, 1921, 1931, 1934, 1935, 1948, 1949,
             1950, 1964, 1965, 1966, 1977, 1979, 
             2002, 2003, 2162, 2163, 2166,  
             2167, 2172, 2173, 2174, 2177, 2183, 2193, 2194, 2195, 2201, 2212, 
             2213, 2214, 2225, 2226, 2227, 2233, 2239, 2254, 2255, 2273, 2279,
             2281, 2290, 2291, 2299, 2306, 2307, 2311, 2326, 2327, 2329, 2343, 
             2345, 2370, 2371, 2882, 2883, 2889, 2897, 2898, 2899, 2902, 2903, 
             2907, 2910, 2911, 2923, 2924, 2925, 2926, 2934, 2935, 2948, 2949, 
             2950, 2955, 2966, 2967, 2969, 2977, 2983, 2985, 3004, 3005, 3006, 
             3030, 3031, 3414, 3415, 3419, 3422, 3423, 3433, 3435, 3436, 3437,
             3438, 3439, 3449, 3460, 3461, 3462, 3465, 3467, 3490, 3491, 3492, 
             3493, 3494, 3515, 3516, 3517, 3518, 3534, 3535, 3556, 3557, 3558,
             3591, 4324, 4325, 4326, 4332, 4333, 4334, 4344, 4346, 4348, 4349,
             4353, 4354, 4355, 4361, 4363, 4366, 4367, 4385, 4386, 4387, 4388,
             4389, 4390, 4393, 4401, 4402, 4403, 4424, 4426, 4428, 4429, 4433,
             4450, 4451, 4452, 4453, 4454, 4466, 4467, 4475, 4478, 4479, 4508,
             4509, 4510, 4545, 4546, 4547 \\
$(73 , 46)$ & 487, 649, 859, 865, 1071, 1145, 1153, 1351, 1607, 1801, 2023, 2025,
             2027, 2034, 2035, 2043, 2049, 2279, 2281, 2290, 2291, 2299, 2306,
             2307, 2401, 2411, 3035, 3038, 3039, 3041, 3051, 3052, 3053, 3054,
             3065, 3074, 3075, 3602, 3603, 3617, 3643, 4553, 4558, 4559, 4562,
             4563, 4577, 4580, 4581, 4582, 4585, 4598, 4599, 4612, 4613, 4614 \\
$(92 , 58)$ & 3567, 4491, 4513, 4521, 4551 \\
\end{longtable}

\vspace{20pt}

\textsc{Institut de Physique du Globe de Paris, Université Paris Cité, France}

{\it E-mail: }{\tt rozier$@$ipgp.fr}

\hspace{10pt}

\textsc{Brié et Angonnes, France}

{\it E-mail: }{\tt claudeterracol34$@$gmail.com}

\end{document}